\documentclass[11pt,reqno]{amsart}
\usepackage{amssymb}
\usepackage{eurosym}
\usepackage{amsfonts,amssymb}
\usepackage{amsmath}
\usepackage{graphicx}
\usepackage{amscd}
\usepackage{fullpage}
\usepackage{amsfonts}
\usepackage{mathrsfs}
\usepackage{float}

\restylefloat{figure}
\newtheorem{theorem}{Theorem}

\newtheorem{definition}{Definition}
\newtheorem{example}{Example}

\newtheorem{lemma}{Lemma}

\numberwithin{equation}{section}

\begin{document}
\title[Oscillation of Equations with Non-monotone Arguments]
{Iterative
oscillation tests for differential equations \\
with several non-monotone arguments}
\author{Elena Braverman$^1$}
\address{Department of Mathematics and Statistics\\
University of Calgary\\
2500 University Drive N. W., Calgary, Canada AB T2N 1N4}
\email{maelena@math.ucalgary.ca} 
\author{George E. Chatzarakis}
\address{Department of Electrical and Electronic Engineering Educators\\
School of Pedagogical and Technological Education (ASPETE)\\
14121, N. Heraklio, Athens, Greece}
\email{geaxatz@otenet.gr, gea.xatz@aspete.gr}
\author{Ioannis P. Stavroulakis} 
\address{Department of Mathematics\\
University of Ioannina\\
451 10 Ioannina, Greece}
\email{ipstav@uoi.gr}

\begin{abstract}
Sufficient oscillation conditions involving $\limsup $ and $\liminf $ for
first-order differential equations with several non-monotone deviating
arguments and nonnegative coefficients are obtained. The results are based
on the iterative application of the Gr\"{o}nwall inequality. Examples
illustrating the significance of the results are also given. 




\end{abstract}

\maketitle

\section{Introduction}

\footnotetext[1]{Corresponding author. E-mail maelena@math.ucalgary.ca, phone (403)-220-3956, fax 
(403)-282-5150} 

In this paper we consider the differential equation with several variable
deviating arguments of either delay 
\begin{equation}  \label{E_R}
x^{\prime}(t)+\sum_{i=1}^{m}p_{i}(t)x(\tau _{i}(t))=0,\;\; t \geq 0
\end{equation}
or advanced type 
\begin{equation}
x^{\prime}(t)-\sum_{i=1}^{m} p_{i}(t)x(\sigma _{i}(t))=0, \;\; t \geq 0.
\label{E_A}
\end{equation}%
Equations \eqref{E_R} and (\ref{E_A}) are studied under the following
assumptions: everywhere $p_{i}(t) \geq 0$, $1\leq i\leq m$, $t \geq 0$, $%
\tau _{i}(t)$, $1\leq i\leq m$, are Lebesgue measurable functions satisfying 
\begin{equation}
\tau _{i}(t)\leq t,\text{ \ \ }\forall t \geq 0 \text{ \ \ \ and \ \ \ }%
\lim\limits_{t\rightarrow \infty }\tau_{i}(t)=\infty, \ \ \ 1\leq i\leq m
\label{1.1}
\end{equation}
and 
\begin{equation}
\sigma _{i}(t)\geq t,\ \ t \geq 0 \text{, \ \ \ }1\leq i\leq m,  \label{1.2}
\end{equation}
respectively. In addition, we consider the initial condition for \eqref{E_R} 
\begin{equation}  \label{initial}
x(t)= \varphi(t), \; \; t \leq 0,
\end{equation}
where $\varphi:(-\infty,0] \to {\mathbb{R}}$ is a bounded Borel measurable
function.

\begin{definition}
\textbf{A solution of \eqref{E_R}, \eqref{initial} } is an absolutely
continuous on $[0,\infty)$ function satisfying \eqref{E_R} for almost all $t \geq 0$
and \eqref{initial} for all $t \leq 0$. By \textbf{a solution of \eqref{E_A} 
} we mean an absolutely continuous on $[0,\infty)$ function satisfying \eqref{E_A} for almost all $t \geq 0$.
\end{definition}

In the special case $m=1$ equations \eqref{E_R} and \eqref{E_A} reduce to
the form 
\begin{equation}
x^{\prime }(t)+p(t)x(\tau (t))=0\text{, \ \ \ }t\geq 0  \label{1.3}
\end{equation}%
and 
\begin{equation}
x^{\prime }(t)-p(t)x(\sigma (t))=0\text{, \ \ \ }t\geq 0,  \label{1.4}
\end{equation}%
respectively.

\begin{definition}
A solution $x(t)$ of \eqref{E_R} \textrm{{(}or \eqref{E_A}{)} is \textbf{%
oscillatory} if it is neither eventually positive nor eventually negative.
If there exists an eventually positive or an eventually negative solution,
the equation is \textbf{nonoscillatory}. An equation is \textbf{oscillatory}
if all its solutions oscillate. }
\end{definition}

In the last few decades, oscillatory behavior and stability of first-order
differential equations with deviating arguments have been extensively
studied, see, for example, papers \cite{2}-\cite{5}, \cite{8}-\cite{16}, 
\cite{17}-\cite{21} and references cited therein. For the general
oscillation theory of differential equations the reader is referred to the
monographs \cite{ABBD,6,7,17a}.

In 1978, Ladde \cite{16} and in 1982, Ladas and Stavroulakis \cite{15}
proved that if 
\begin{equation}
\liminf_{t\rightarrow \infty }\int\limits_{\tau _{\max
}(t)}^{t}\sum\limits_{i=1}^{m}p_{i}(s)ds>\frac{1}{e}\,,  \label{1.8}
\end{equation}%
where $\tau _{\max }(t)=\max_{1\leq i\leq m}\{\tau _{i}(t)\},$ then all
solutions of \eqref{E_R} oscillate, while if 
\begin{equation}
\liminf_{t\rightarrow \infty }\int\limits_{t}^{\sigma _{\min
}(t)}\sum\limits_{i=1}^{m}p_{i}(s)ds>\frac{1}{e}\,,  \label{1.9}
\end{equation}%
where $\displaystyle\sigma _{\min }(t)=\min_{1\leq i\leq m}\{\sigma
_{i}(t)\},$ then all solutions of \eqref{E_A} oscillate. See also \cite[%
Theorem 2.7.1]{17a} and \cite[Theorem 1$^{\prime }$]{9}.

In 1984, Hunt and Yorke \cite{10} proved that if $t-\tau _{i}(t)\leq \tau_{0}
$ for some $\tau_{0}>0$, $1\leq i\leq m,$ and 
\begin{equation}
\liminf_{t\rightarrow \infty }\sum\limits_{i=1}^{m}p_{i}(t)\left(
t-\tau_{i}(t)\right) >\frac{1}{e} \, ,  \label{1.10}
\end{equation}%
then all solutions of \eqref{E_R} oscillate.

In 1990, Zhou \cite{21} proved that if $\sigma _{i}(t)-t\leq \sigma_{0}$ for
some $\sigma_{0}>0$, $1\leq i\leq m,$ and 
\begin{equation}
\liminf_{t\rightarrow \infty }\sum\limits_{i=1}^{m}p_{i}(t)\left( \sigma
_{i}(t)-t\right) > \frac{1}{e} \, ,  \label{1.11}
\end{equation}%
then all solutions of \eqref{E_A} oscillate. See also this result in the
monograph \cite[Corollary 2.6.12]{7}.

For differential equation \eqref{1.3} with one delay, in 2011 Braverman and
Karpuz \cite{3} established the following theorem in the case that the
argument $\tau (t)$ is non-monotone 
and $g(t)$ is defined as 
\begin{equation*}
g(t)=\sup_{s\leq t}\tau (s),\ \ \ \ t\geq 0.
\end{equation*}

\begin{theorem}
\label{theorem1} Assume that \eqref{1.1} holds and 
\begin{equation}
\limsup\limits_{t\rightarrow \infty }\int\limits_{g(t)}^{t}p(s)\exp \left\{
\int\limits_{\tau (s)}^{g(t)}p(\xi )d\xi \right\} ds>1.  \label{1.12}
\end{equation}
Then all solutions of \eqref{1.3} oscillate.
\end{theorem}

In 2014, Theorem~\ref{theorem1} was improved by Stavroulakis [21] as follows:

\begin{theorem}
Assume that \eqref{1.1} holds,%
\begin{equation*}
0<\alpha :=\liminf_{t\rightarrow \infty }\int\limits_{\tau
(t)}^{t}p(s)ds\leq \frac{1}{e}
\end{equation*}%
and 
\begin{equation*}
\limsup_{t\rightarrow \infty }\int_{g(t)}^{t}p(s)\exp \left\{ \int_{\tau
(s)}^{g(t)}p(\xi )d\xi \right\} ~ds>1-\frac{1-\alpha -\sqrt{1-2\alpha
-\alpha ^{2}}}{2}.
\end{equation*}%
Then all solutions of \eqref{1.3} oscillate.
\end{theorem}

In 2015, Chatzarakis and \"{O}calan \cite{4} established the following
theorem in the case that the arguments $\sigma _{i}(t)$, $1\leq i\leq m$ are
non-monotone 
and $\rho _{i}(t)=\inf_{s\geq t}\sigma _{i}(s),$ $t\geq 0$, $\rho
(t)=\min_{1\leq i\leq m}\rho _{i}(t),$ $t\geq 0$:

\begin{theorem}
\label{theorem3}
Assume that \eqref{1.2} holds, and either 
\begin{equation}
\limsup\limits_{t\rightarrow \infty
}\int\limits_{t}^{\rho(t)}\sum\limits_{i=1}^{m}p_{i}(s)\exp \left\{
\sum\limits_{j=1}^{m}\int\limits_{\rho _{i}(t)}^{\sigma _{i}(s)}p_{j}(\xi
)d\xi \right\} ds>1\text{,}  \label{1.13}
\end{equation}%
or 
\begin{equation}
\liminf_{t\rightarrow \infty }\int\limits_{t}^{\rho
(t)}\sum\limits_{i=1}^{m}p_{i}(s)\exp \left\{
\sum\limits_{j=1}^{m}\int\limits_{\rho _{i}(t)}^{\sigma _{i}(s)}p_{j}(\xi
)d\xi \right\} ds>\frac{1}{e}\text{.}  \label{1.14}
\end{equation}%
Then all solutions of \eqref{E_A} oscillate.
\end{theorem}

In addition to purely mathematical interest, consideration of non-monotone
arguments is important, since it approximates the natural phenomena
described by equations of the type of \eqref{E_R} or \eqref{E_A}. In fact,
there are always natural disturbances (e.g. noise in communication systems)
that affect all the parameters of the equation and therefore monotone
arguments will generally become non-monotone. In view of this, it is
interesting to consider the case where the arguments (delays and advances)
are non-monotone. In the present paper we obtain sufficient oscillation
conditions involving $\limsup $ and $\liminf $. 

\section{Main Results}

In this section, we establish sufficient oscillation conditions for %
\eqref{E_R} and \eqref{E_A} satisfying \eqref{1.1} and \eqref{1.2},
respectively. The method we apply is based on the iterative construction of
solution estimates and repetitive application of the Gr\"{o}nwall
inequality. It also uses some ideas of \cite{12}, where some oscillation
results for a differential equation with a single delay were established.

\subsection{Delay equations}

Let 
\begin{equation}
g _{i}(t)=\sup_{0\leq s\leq t}\tau _{i}(s),\text{ \ \ }t \geq 0  \label{2.3}
\end{equation}
and 
\begin{equation}
g(t)=\max_{1\leq i\leq m} g_{i}(t),\text{ \ \ }t\geq 0\text{.}  \label{2.4}
\end{equation}
As follows from their definitions, the functions $g_i(t)$, $1\leq i \leq m$
and $g(t)$ are non-decreasing Lebesgue measurable functions satisfying $g
(t)\leq t$, $g_{i}(t)\leq t$, $1\leq i\leq m$ for all $t \geq 0$.

The following lemma provides an estimation for a rate of decay for a
positive solution. Such estimates are a basis for most oscillation
conditions.

\begin{lemma}
\label{lemma2.1} Assume that $x(t)$ is a positive solution of \eqref{E_R}.
Denote 
\begin{equation}
a_{1}(t,s):=\exp \left\{ \int_{s}^{t}\sum_{i=1}^{m}p_{i}(\zeta )~d\zeta
\right\}  \label{2.7}
\end{equation}%
and 
\begin{equation}
a_{r+1}(t,s):=\exp \left\{ \int_{s}^{t}\sum_{i=1}^{m}p_{i}(\zeta
)a_{r}(\zeta ,\tau _{i}(\zeta ))~d\zeta \right\} ,\;\;r\in {\mathbb{N}}.
\label{2.8}
\end{equation}%
Then 
\begin{equation}
x(t)a_{r}(t,s)\leq x(s),\text{ \ \ \ \ }0\leq s\leq t.  \label{2.9}
\end{equation}
\end{lemma}
\begin{proof}
The function $x(t)$ is a positive solution of \eqref{E_R} for any $t$, so 
\begin{equation*}
x^{\prime }(t)=-\sum_{i=1}^{m}p_{i}(t)x(\tau _{i}(t))\leq 0,\ \ \ \ t\geq 0,
\end{equation*}%
which means that the solution $x(t)$ is monotonically decreasing. Thus $%
x(\tau _{i}(t))\geq x(t)$ and 
\begin{equation*}
x^{\prime }(t)+x(t)\sum_{i=1}^{m}p_{i}(t)\leq 0\text{, \ \ \ \ }t\geq 0.
\end{equation*}%
Applying the Gr\"{o}nwall inequality, we obtain 
\begin{equation*}
x(t)\leq x(s)\exp \left\{ -\int_{s}^{t}\sum_{i=1}^{m}p_{i}(\zeta )~d\zeta
\right\} ,\ \ \ \ 0\leq s\leq t,
\end{equation*}%
or 
\begin{equation*}
x(t)\exp \left\{ \int_{s}^{t}\sum_{i=1}^{m}p_{i}(\zeta )~d\zeta \right\}
\leq x(s),\ \ \ \ 0\leq s\leq t,
\end{equation*}%
that is, estimate \eqref{2.9} is valid for $r=1$.

Next, let us proceed to the induction step: assume that \eqref{2.9} holds
for some $r>1$, then 
\begin{equation}
x(t)a_{r}(t,\tau _{i}(t))\leq x(\tau _{i}(t)).  \label{2.6}
\end{equation}%
Substituting \eqref{2.6} into \eqref{E_R} leads to the estimate 
\begin{equation*}
x^{\prime }(t)+x(t)\sum_{i=1}^{m}p_{i}(t)a_{r}(t,\tau _{i}(t))\leq 0.
\end{equation*}%
Again, applying the Gr\"{o}nwall inequality, we have%
\begin{equation*}
x(t)\leq x(s)\exp \left\{ -\int_{s}^{t}\sum_{i=1}^{m}p_{i}(\zeta
)a_{r}(\zeta ,\tau _{i}(\zeta ))~d\zeta \right\} ,
\end{equation*}%
or 
\begin{equation*}
x(t)\exp \left\{ \int_{s}^{t}\sum_{i=1}^{m}p_{i}(\zeta )a_{r}(\zeta ,\tau
_{i}(\zeta ))~d\zeta \right\} \leq x(s),
\end{equation*}
that is, 
\begin{equation*}
x(t)a_{r+1}(t,s)\leq x(s),
\end{equation*}%
which completes the induction step and the proof of the lemma.
\end{proof}

Let us illustrate how the estimate developed in Lemma~\ref{lemma2.1} works
in the case of autonomous equations. The series of estimates is evaluated
using computer tools, which recently became an efficient tool in
computer-assisted proofs \cite{1}. We suggest that, similarly, a computer
algebra can be used to construct the estimate iterates and, ideally, the
limit estimate. The example below illustrates the procedure.

\begin{example}
\label{example2.1} The equation 
\begin{equation*}
x^{\prime}(t)+ \alpha e^{-\alpha} x(t-1)=0,\;\; t \geq 0, \;\; \alpha \geq 0
\end{equation*}
has an exact nonoscillatory solution $e^{-\alpha t}$. For $\alpha=0.5$ the
exact rate of decay (up to the sixth digit after the decimal point) is $%
x(t+1) \approx 0.606531 x(t)$, while $a_1^{-1}(t,t-1) \approx 0.738403$, $%
a_2^{-1}(t,t-1) \approx 0.663183$, $a_{10}^{-1}(t,t-1) \approx 0.606725$, $%
a_{18}^{-1}(t,t-1) \approx 0.606531$. The largest value of the coefficient
of $1/e$ is attained at $\alpha=1$; it is well known that it is the maximal
coefficient when the equation is still nonoscillatory. The decay of the
estimate $x(t+1)\leq \frac{1}{e} x(t) \approx 0.367879 x(t)$ is the slowest: 
$a_1^{-1}(t,t-1) \approx 0.692201$, $a_2^{-1}(t,t-1) \approx 0.587744$, $%
a_{10}^{-1}(t,t-1) \approx 0.430949$, $a_{50}^{-1}(t,t-1) \approx 0.381994$, 
$a_{100}^{-1}(t,t-1) \approx 0.375068$, $a_{1000}^{-1}(t,t-1) \approx
0.368613$.
\end{example}

\begin{theorem}
\label{theorem2.4} 
Let $p_{i}(t)\geq 0$, $1\leq i\leq m$, and $g(t)$ be
defined by \eqref{2.4}, while $a_{r}(t,s)$ by \eqref{2.7},\eqref{2.8}. If %
\eqref{1.1} holds and for some $r\in \mathbb{N}$ 
\begin{equation}
\limsup_{t\rightarrow \infty }\int_{g(t)}^{t}\sum_{i=1}^{m}p_{i}(\zeta
)a_{r}(g(t),\tau _{i}(\zeta ))~d\zeta >1,  \label{2.10}
\end{equation}
then all solutions of \eqref{E_R} oscillate.
\end{theorem}

\begin{proof}
Assume, for the sake of contradiction, that there exists a nonoscillatory
solution $x(t)$ of \eqref{E_R}. Since $-x(t)$ is also a solution of %
\eqref{E_R}, we can consider only the case when the solution $x(t)$ is
eventually positive. Then there exists $t_{1}>0$ such that $x(t)>0$ and $%
x\left( \tau _{i}(t)\right)>0$, for all $t\geq t_{1}$. Thus, from \eqref{E_R}
we have 
\begin{equation*}
x^{\prime }(t)=-\sum_{i=1}^{m}p_{i}(t)x(\tau _{i}(t))\leq 0,\text{ \ \ for
all }t\geq t_{1}\text{,}
\end{equation*}
which means that $x(t)$ is an eventually non-increasing positive function.

Integrating \eqref{E_R} from $g(t)$ to $t$, and using the fact that the
function $x$ is non-increasing, while the function $g$ defined by \eqref{2.4}
is non-decreasing, and taking into account that 
\begin{equation*}
\tau _{i}(t)\leq g(t)\text{ \ \ and \ \ }x(\tau _{i}(s))\geq
x(g(t))a_{r}(g(t),\tau _{i}(s)),
\end{equation*}%
we obtain, for sufficiently large $t$, 
\begin{eqnarray*}
x(g(t)) &=&x(t)+\int_{g(t)}^{t}\sum_{i=1}^{m}p_{i}(\zeta )x(\tau _{i}(\zeta
))~d\zeta \\
& > &\int_{g(t)}^{t}\sum_{i=1}^{m}p_{i}(\zeta )x(\tau _{i}(\zeta ))~d\zeta \\
&\geq &x(g(t))\int_{g(t)}^{t}\sum_{i=1}^{m}p_{i}(\zeta )a_{r}(g(t),\tau
_{i}(\zeta ))~d\zeta .
\end{eqnarray*}%
Hence 
\begin{equation*}
x(g(t))\left[ 1-\int_{g(t)}^{t}\sum_{i=1}^{m}p_{i}(\zeta )a_{r}(g(t),\tau
_{i}(\zeta ))~d\zeta \right] \geq 0,
\end{equation*}%
which implies 
\begin{equation*}
\limsup_{t\rightarrow \infty }\int_{g(t)}^{t}\sum_{i=1}^{m}p_{i}(\zeta
)a_{r}(g(t),\tau _{i}(\zeta ))~d\zeta \leq 1.
\end{equation*}%
The last inequality contradicts \eqref{2.10}, and the proof is complete.
\end{proof}

The following example illustrates the significance of the condition $%
\lim\limits_{t \rightarrow \infty }\tau_{i}(t)=\infty $, $1\leq i\leq m$, in
Theorem~\ref{theorem2.4}.

\begin{example}
\label{example2} Consider the delay differential equation \eqref{1.3} with 
\begin{equation*}
p(t)\equiv 2,\;\;\tau (t)=\left\{ 
\begin{array}{ll}
-1, & \text{\ \ if }t\in \lbrack 2k,2k+1), \\ 
t, & \text{\ \ if }t\in \lbrack 2k+1,2k+2),%
\end{array}%
\right. \text{\ \ \ }k\in \mathbb{N}_{0}=\{ 0,1,2,\dots \} \;.
\end{equation*}%
By \eqref{2.4}, we find 
\begin{equation*}
g(t)=\sup_{0\leq s\leq t}\tau (s)=\left\{ 
\begin{array}{ll}
\lbrack t], & \text{\ \ if }t\in \lbrack 2k,2k+1), \\ 
t, & \text{\ \ if }t\in \lbrack 2k+1,2k+2),%
\end{array}%
\right. k\in \mathbb{N}_{0}.
\end{equation*}%
If $t=2k+0.8$, then $g(t)=[t]=2k$ and 
\begin{equation*}
\int_{g(t)}^{t}p(\zeta )~d\zeta =\int_{g(2k+0.8)}^{2k+0.8}p(\zeta )~d\zeta
=2\int_{2k}^{2k+0.8}~d\zeta =1.6>1,
\end{equation*}%
which means that \eqref{2.10} is satisfied for any $r$.

However, equation \eqref{1.3} has a nonoscillatory solution 
\begin{equation*}
x(t)=\varphi(t)=t+1, ~~t \in [-1,0], \;\; x(t)= \left\{ 
\begin{array}{ll}
e^{-2[t]}, & \text{\ \ if } t \in [2k,2k+1), \\ 
e^{-2(t-k-1)}, & \text{\ \ if } t \in [2k+1,2k+2),%
\end{array}
\right.
\end{equation*}
which illustrates the significance of the condition $\lim\limits_{t
\rightarrow \infty }\tau (t)=\infty $ in Theorem~\ref{theorem2.4}. 
\end{example}

In 1992, Yu et al. \cite{20a} proved the following result.

\begin{lemma}
\label{lemma_for_max_decrease} In addition to the hypothesis (1.3), assume
that $g(t)$ is defined by \eqref{2.4}, 
\begin{equation}
0<\alpha :=\liminf_{t\rightarrow
\infty}\int_{g(t)}^{t}\sum_{i=1}^{m}p_{i}(s)\,ds\leq \frac{1}{e},
\label{alpha_def}
\end{equation}%
and $x(t)$ is an eventually positive solution of \eqref{E_R}. Then 
\begin{equation}
\liminf_{t\rightarrow \infty }\frac{x(t)}{x(g(t))}\geq \frac{1-\alpha -\sqrt{%
1-2\alpha -\alpha ^{2}}}{2},  \label{add_del_1}
\end{equation}
\end{lemma}

Based on inequality \eqref{add_del_1}, we establish the following theorem.

\begin{theorem}
\label{theorem2.4_abc} Assume that $p_{i}(t)\geq 0$, $1\leq i\leq m$, $g(t)$
is defined by \eqref{2.4}, $a_{r}(t,s)$ by \eqref{2.8},\eqref{2.7}and (2.8)
holds. If for some $r\in {\mathbb{N}}$ 
\begin{equation}
\limsup_{t\rightarrow \infty }\int_{g(t)}^{t}\sum_{i=1}^{m}p_{i}(\zeta
)a_{r}(g(t),\tau _{i}(\zeta ))~d\zeta >1-\frac{1-\mathfrak{\alpha }-\sqrt{1-2%
\mathfrak{\alpha }-\mathfrak{\alpha }^{2}}}{2},  \label{add_del_2}
\end{equation}
then all solutions of \eqref{E_R} oscillate.
\end{theorem}

\begin{proof}
Assume, for the sake of contradiction, that there exists a nonoscillatory
solution $x(t)$ of \eqref{E_R}. Then, as in the proof of Theorem~\ref%
{theorem2.4}, we obtain, for sufficiently large $t$, 
\begin{eqnarray*}
x(g(t)) &=&x(t)+\int_{g(t)}^{t}\sum_{i=1}^{m}p_{i}(\zeta )x(\tau
_{i}(\zeta))~d\zeta \\
&\geq &x(t)+x(g(t))\int_{g(t)}^{t}\sum_{i=1}^{m}p_{i}(\zeta )a_{r}(g(t),\tau
_{i}(\zeta ))~d\zeta .
\end{eqnarray*}%
That is, 
\begin{equation*}
\int_{g(t)}^{t}\sum_{i=1}^{m}p_{i}(\zeta )a_{r}(g(t),\tau _{i}(\zeta
))~d\zeta \leq 1-\frac{x(t)}{x(g(t))},
\end{equation*}%
which gives 
\begin{equation*}
\limsup_{t\rightarrow \infty }\int_{g(t)}^{t}\sum_{i=1}^{m}p_{i}(\zeta
)a_{r}(g(t),\tau _{i}(\zeta ))~d\zeta \leq 1-\liminf_{t\rightarrow \infty }%
\frac{x(t)}{x(g(t))}.
\end{equation*}%
Taking into account that \eqref{add_del_1} holds, the last inequality leads
to 
\begin{equation*}
\limsup_{t\rightarrow \infty }\int_{g(t)}^{t}\sum_{i=1}^{m}p_{i}(\zeta
)a_{r}(g(t),\tau _{i}(\zeta ))~d\zeta \leq 1-\frac{1-\mathfrak{\alpha }-%
\sqrt{1-2\mathfrak{\alpha }-\mathfrak{\alpha }^{2}}}{2}\text{,}
\end{equation*}%
which contradicts condition \eqref{add_del_2}.

The proof of the theorem is complete.
\end{proof}


Next, let us proceed to an oscillation condition involving $\liminf$.

\begin{theorem}
\label{theorem3.3} Assume that $p_{i}(t) \geq 0$, $1\leq i\leq m$, %
\eqref{1.1} holds and $a_{r}(t,s)$ are defined by \eqref{2.7},\eqref{2.8}.
If for some $r\in \mathbb{N}$ 
\begin{equation}
\liminf_{t\rightarrow \infty }\int_{g(t)}^{t}\sum_{i=1}^{m}p_{i}(\zeta)
a_{r}(g(\zeta),\tau _{i}(\zeta ))\,d\zeta >\frac{1}{e},  \label{3.3}
\end{equation}
then all solutions of \eqref{E_R} oscillate.
\end{theorem}

\begin{proof}
Assume, for the sake of contradiction, that there exists a nonoscillatory
solution $x(t)$ of \eqref{E_R}. Similarly to the proof of Theorem~\ref%
{theorem2.4}, we can confine our discussion only to the case of $x(t)$ being
eventually positive. Then there exists $t_{1}> 0$ such that $x(t)>0$ and $%
x\left( \tau _{i}(t)\right) >0$ for all $t\geq t_{1}$. Thus, from \eqref{E_R}
we have 
\begin{equation*}
x^{\prime }(t)=-\sum_{i=1}^{m}p_{i}(t)x(\tau _{i}(t))\leq 0,\text{ \ \ for
all }t\geq t_{1},
\end{equation*}
which means that $x(t)$ is an eventually non-increasing positive function.

For $t\geq t_{1}$, \eqref{E_R} can be rewritten as 
\begin{equation*}
\frac{x^{\prime }(t)}{x\left( t\right) }+\sum_{i=1}^{m}p_{i}(t)\frac{x\left(
\tau _{i}(t)\right) }{x\left( t\right) }=0,\text{ \ \ for all }t\geq t_{1}%
\text{.}
\end{equation*}
Integrating from $g(t)$ to $t$ gives 
\begin{equation*}
\ln \left( \frac{x\left( t\right) }{x\left( g(t)\right) }\right)
+\int\limits_{g(t)}^{t}\sum_{i=1}^{m}p_{i}(\zeta )\frac{x\left( \tau
_{i}(\zeta )\right) }{x\left( \zeta \right) }d\zeta =0\text{ \ \ for all }%
t\geq t_{2}\geq t_{1}\text{.}
\end{equation*}

Since $g(t)\geq \tau _{i}(\zeta )$, by Lemma~\ref{lemma2.1} we have $x(\tau
_{i}(\zeta))\geq a_{r}(g(t),\tau _{i}(\zeta ))x(g(t))$, and therefore 
\begin{equation*}
\ln \left( \frac{x(t)}{x(g(t))}\right)
+\int_{g(t)}^{t}\sum_{i=1}^{m}p_{i}(\zeta )a_{r}(g(t),\tau _{i}(\zeta ))%
\frac{x(g(t))}{x(\zeta )}\,d\zeta \leq 0. 
\end{equation*}
In view of $x(g(t))\geq x(\zeta )$, the last inequality becomes 
\begin{equation}
\ln \left( \frac{x(t)}{x(g(t))}\right)
+\int_{g(t)}^{t}\sum_{i=1}^{m}p_{i}(\zeta )a_{r}(g(t),\tau _{i}(\zeta
))\,d\zeta \leq 0\text{.}  \label{2.12}
\end{equation}

Also, from \eqref{3.3} it follows that there exists a constant $c>0$ such
that for some $t_3 \geq t_2$ 
\begin{equation}
\label{2.13_abc}
\int_{g(t)}^{t}\sum_{i=1}^{m} p_{i}(\zeta) a_{r}(g(\zeta),\tau _{i}(\zeta
))\,d\zeta \geq c >\frac{1}{e}, \quad t \geq t_3 \geq t_2.
\end{equation}
For a fixed $s$, the function $a_r(t,s)$ is non-decreasing it $t$, $g$ is also non-decreasing, 
therefore for $\zeta \leq t$,
$a_{r}(g(t),\tau _{i}(\zeta)) \geq a_{r}(g(\zeta),\tau _{i}(\zeta))$. Hence
\begin{equation}
\int_{g(t)}^{t}\sum_{i=1}^{m} p_{i}(\zeta) a_{r}(g(t),\tau _{i}(\zeta
))\,d\zeta \geq c >\frac{1}{e}, \quad t \geq t_3 \geq t_2.  \label{2.13new}
\end{equation}
Combining inequalities \eqref{2.12} and \eqref{2.13new}, we obtain 
\begin{equation*}
\ln \left( \frac{x(t)}{x(g(t))} \right) +c \leq 0, \quad t \geq t_3.
\end{equation*}
Thus 
\begin{equation*}
\frac{x(g(t))}{x(t)}\geq e^{c}\geq ec>1,
\end{equation*}
which implies for some $t \geq t_4 \geq t_3$ 
\begin{equation*}
(ec)x(t) \leq x(g(t)). 
\end{equation*}
Repeating the above argument leads to a new estimate $x(g(t))/x(t)>(ec)^{2}$, for $t$ large enough. 
Continuing by induction, we get 
\begin{equation*}
\frac{x(g(t))}{x(t)} \geq (ec)^{k}, \mbox{~~ for sufficiently large ~~}t, 
\end{equation*}
where $ec>1$. As $ec>1$, there is $k\in {\mathbb{N}}$ satisfying $k >
2(\ln(2)-\ln(c))/(1+\ln(c))$ such that for $t$ large enough 
\begin{equation}
\frac{x(g(t))}{x(t)}\geq (ec)^{k} > \frac{4}{c^{2}}.  \label{3.3b}
\end{equation}

Further, integrating \eqref{E_R} from $g(t)$ to $t$ yields 
\begin{equation*}
x(g(t)) - x(t)-\int_{g(t)}^t \sum_{i=1}^{m} p_{i}(\zeta) x(\tau_i(\zeta))\,
d \zeta =0.
\end{equation*}
Inequality \eqref{2.9} in Lemma~\ref{lemma2.1} used in the above equality
leads to 
the differential inequality 
\begin{equation*}
x(g(t)) - x(t)- x(g(t)) \int_{g(t)}^t \sum_{i=1}^{m} p_{i}(\zeta)a_r
(g(t),\tau_i(\zeta)) \, d \zeta \geq 0.
\end{equation*}
The strict inequality is valid if we omit $x(t)>0$ in the left-hand side: 
\begin{equation*}
x(g(t))\left [ 1 - \int_{g(t)}^t \sum_{i=1}^{m} p_{i}(\zeta)a_r
(g(t),\tau_i(\zeta))\, d \zeta \right] > 0.
\end{equation*}
From \eqref{2.13_abc}, for large enough $t$,
\begin{equation}
0 < c \leq \int_{g(t)}^{t} \sum_{i=1}^{m} p_{i}(\zeta) a_r
(g(\zeta),\tau_i(\zeta))\, d\zeta < 1.  \label{3.3c}
\end{equation}
Taking the integral on $[g(t),t]$ which is not less than $c$, we split the
interval into two parts where integrals are not less than $c/2$, let $t_m
\in (g(t),t) $ be the splitting point: 
\begin{equation*}
\int_{g(t)}^{t_m} \sum_{i=1}^{m} p_{i}(\zeta) a_r (g(\zeta),\tau_i(\zeta))\,
d\zeta \geq \frac{c}{2}, \quad \int_{t_m}^{t} \sum_{i=1}^{m} p_{i}(\zeta)
a_r (g(\zeta),\tau_i(\zeta))\, d\zeta \geq \frac{c}{2}.
\end{equation*}
Since $g(\zeta) \leq g(t_m)$ in the first integral, 
and $g(\zeta) \leq g(t)$ in the second one, 
we obtain 
\begin{equation}
\int_{g(t)}^{t_m} \sum_{i=1}^{m} p_{i}(\zeta) a_r (g(t_m),\tau_i(\zeta))\,
d\zeta \geq \frac{c}{2}, \quad \int_{t_m}^{t} \sum_{i=1}^{m} p_{i}(\zeta)
a_r (g(t),\tau_i(\zeta))\, d\zeta \geq \frac{c}{2}.  \label{3.3d}
\end{equation}
Integrating \eqref{E_R} from $t_m$ to $t$, along with incorporating the
inequality $x(\tau _{i}(\zeta))\geq a_{r}(g(t),\tau _{i}(\zeta ))x(g(t))$,
gives 
\begin{equation*}
-x(t_m)+x(t) + x(g(t)) \int_{t_m}^{t} \sum_{i=1}^{m} p_{i}(\zeta) a_r
(g(t),\tau_i(\zeta)) \leq 0.
\end{equation*}
Together with the second inequality in \eqref{3.3d}, this implies 
\begin{equation}
x(t_m) \geq \frac{c}{2} x(g(t)).  \label{3.3e}
\end{equation}
%
%
Similarly, integration of \eqref{E_R} from $g(t)$ to $t_m$ with a later
application of Lemma~\ref{lemma2.1} leads to 
\begin{equation*}
x(t_m)-x(g(t)) + x(g(t_m)) \int_{g(t)}^{t_m} \sum_{i=1}^{m} p_{i}(\zeta) a_r
(g(t_m),\tau_i(\zeta))\, d\zeta\leq 0,
\end{equation*}
which together with the first inequality in \eqref{3.3d} yields that 
\begin{equation}
x(g(t)) \geq \frac{c}{2} x(g(t_m))  \label{3.3f}
\end{equation}
Inequalities \eqref{3.3e} and \eqref{3.3f} imply 
\begin{equation*}
x(g(t_m)) \leq \frac{2}{c} x(g(t)) \leq \frac{4}{c^2} x(t_m),
\end{equation*}
which contradicts \eqref{3.3b}. Thus, all solutions of \eqref{E_R} oscillate.
\end{proof}

As non-oscillation of \eqref{E_R} is equivalent to existence of a positive
or a negative solution of the relevant differentiation inequalities (see,
for example, \cite[Theorem 2.1, p. 25]{ABBD}), Theorems~\ref{theorem2.4},~%
\ref{theorem2.4_abc} and~\ref{theorem3.3} lead to the following result. 

\begin{theorem}
\label{theorem7}
Assume that all the conditions of anyone of Theorems~\ref{theorem2.4},~\ref%
{theorem2.4_abc} and~\ref{theorem3.3} hold. Then

(i) the differential inequality 
\begin{equation*}
x^{\prime }(t)+\sum\limits_{i=1}^{m}p_{i}(t)x\left( \tau _{i}(t)\right) \leq
0,~~t\geq 0
\end{equation*}
has no eventually positive solutions;

(ii) the differential inequality 
\begin{equation*}
x^{\prime }(t)+\sum\limits_{i=1}^{m}p_{i}(t)x\left( \tau _{i}(t)\right) \geq
0,~~t\geq 0
\end{equation*}
has no eventually negative solutions.
\end{theorem}

\subsection{Advanced Equations}

Similar oscillation theorems for the (dual) advanced differential equation %
\eqref{E_A} can be derived easily. The proofs of these theorems are omitted,
since they are quite similar to the proofs for the delay equation \eqref{E_R}.

Denote 
\begin{equation}
\rho_{i}(t)=\inf_{s\geq t} \sigma _{i}(s),\text{ \ \ } t\geq 0  \label{2.15}
\end{equation}
and 
\begin{equation}
\rho (t)=\min_{1\leq i\leq m}\rho _{i}(t),\text{ \ \ } t \geq 0\text{.}
\label{2.16}
\end{equation}
Clearly, the functions $\rho (t)$, $\rho _{i}(t)$, $1\leq i \leq m$, are
Lebesgue measurable non-decreasing and $\rho (t) \geq t$, $\rho _{i}(t) \geq
t$, $1\leq i\leq m$ for all $t \geq 0$.

\begin{theorem}
\label{theorem2.4a} Assume that $p_{i}(t)\geq 0$, $1\leq i\leq m$, %
\eqref{1.2} holds, $\rho (t)$ is defined by 
\eqref{2.16} and $b_{r}(t,s)$ are denoted as 
\begin{equation}
b_{1}(t,s):=\exp \left\{ \int_{t}^{s}\sum_{i=1}^{m}p_{i}(\zeta )~d\zeta
\right\}  \label{2.17}
\end{equation}%
and 
\begin{equation}
b_{r+1}(t,s):=\exp \left\{ \int_{t}^{s}\sum_{i=1}^{m}p_{i}(\zeta
)b_{r}(\zeta,\sigma _{i}(\zeta ))~d\zeta \right\} ,\;\;r\in {\mathbb{N}}.
\label{2.18}
\end{equation}%
If for some $r\in 
\mathbb{N}
$%
\begin{equation}
\limsup_{t\rightarrow \infty }\int_{t}^{\rho (t)}\sum_{i=1}^{m}p_{i}(\zeta
)b_{r}(\rho (t),\sigma _{i}(\zeta ))\,d\zeta >1,  \label{2.19}
\end{equation}%
then all solutions of \eqref{E_A} oscillate.
\end{theorem}

We would like to mention that Lemma 2 can be extented to the advanced type
differential equation \eqref{E_A} (cf. \cite[Section 2.6.6]{7}).

\begin{lemma}
\label{lemma_for_max_decr1} In addition to hypothesis \eqref{1.2},
assume that $\rho (t)$ is defined by \eqref{2.16}, 
\begin{equation}  \label{2.25}
0<\alpha :=\liminf_{t\rightarrow \infty }\int_{t}^{\rho
(t)}\sum_{i=1}^{m}p_{i}(s)\,ds\leq \frac{1}{e},
\end{equation}%
and $x(t)$ is an eventually positive solution of \eqref{E_A}. Then 
\begin{equation*}
\liminf_{t\rightarrow \infty }\frac{x(t)}{x(\rho (t))}\geq \frac{1-\mathfrak{%
\alpha }-\sqrt{1-2\mathfrak{\alpha }-\mathfrak{\alpha }^{2}}}{2}.
\end{equation*}
\end{lemma}

Based on the above inequality, we establish the following theorem.

\begin{theorem}
\label{theorem2.4ab} Assume that $p_{i}(t)\geq 0$, $1\leq i\leq m$, %
\eqref{1.2} is satisfied, $\rho (t)$ is defined by \eqref{2.16}, $b_{r}(t,s)$
by \eqref{2.17} and \eqref{2.18}, and \eqref{2.25} holds. If for some $r\in 
\mathbb{N}$ 
\begin{equation}
\limsup_{t\rightarrow \infty }\int_{t}^{\rho
(t)}\sum_{i=1}^{m}p_{i}(\zeta)b_{r}(\rho (t),\sigma _{i}(\zeta )) ~d\zeta >1-%
\frac{1-\mathfrak{\alpha }-\sqrt{1-2\mathfrak{\alpha }-\mathfrak{\alpha }^{2}%
}}{2},  \label{add_tag_2}
\end{equation}
then all solutions of \eqref{E_A} oscillate.
\end{theorem}

\begin{theorem}
\label{theorem2.5b} Assume that $p_{i}(t)\geq 0$, $1\leq i\leq m$, %
\eqref{1.2} holds, $\rho (t)$ is defined by 
\eqref{2.16}, $b_{r}(t,s)$ are denoted in \eqref{2.17}, \eqref{2.18}. If for
some $\ r\in {\mathbb{N}}$ 
\begin{equation}
\liminf_{t\rightarrow \infty }\sum_{i=1}^{m}\int_{t}^{\rho (t)}p_{i}(\zeta
)b_{r}(\rho (\zeta),\sigma _{i}(\zeta ))\,d\zeta >\frac{1}{e}\,,  \label{2.19abc}
\end{equation}%
then all solutions of \eqref{E_A} oscillate.
\end{theorem}

A slight modification in the proofs of Theorems~\ref{theorem2.4a}, \ref%
{theorem2.4ab} and \ref{theorem2.5b} leads to the following result about
advanced differential inequalities. 

\begin{theorem}
\label{theorem2.5ab} Assume that all the conditions of anyone of Theorems %
\ref{theorem2.4a}, \ref{theorem2.4ab} and \ref{theorem2.5b} hold. Then

(i) the differential inequality 
\begin{equation*}
x^{\prime }(t)-\sum\limits_{i=1}^{m}p_{i}(t)x\left( \sigma _{i}(t)\right)
\geq 0,~~t\geq 0,
\end{equation*}
has no eventually positive solutions;

(ii) the differential inequality 
\begin{equation*}
x^{\prime }(t)-\sum\limits_{i=1}^{m}p_{i}(t)x\left( \sigma _{i}(t)\right)
\leq 0,\text{ }t\geq 0,
\end{equation*}
has no eventually negative solutions.
\end{theorem}

\section{Examples}

In this section we provide two examples illustrating 
Theorems~\ref{theorem2.4} and \ref{theorem2.4a}.  Similarly, examples to
illustrate the other main results of the paper can be constructed.

\begin{example}
\label{ex_theorem3} Consider the delay differential equation 
\begin{equation}
x^{\prime }(t)+\frac{1}{2e}x(\tau _{1}(t))+\frac{1}{2.2e}x(\tau _{2}(t))=0,%
\text{ \ \ \ }t\geq 1\text{,}  \label{ex_3_eq3.1}
\end{equation}
where (see Fig.~\ref{figure1}, a) 
\begin{equation*}
\tau _{1}(t)=\left\{ 
\begin{array}{ll}
-t+4k+1, & \text{if }t\in \left[ 2k+1,2k+2\right], \\ 
3t-4k-7, & \text{if }t\in \left[ 2k+2,2k+3\right],
\end{array}
\right. \text{ \ and \ }\tau _{2}(t)=\tau _{1}(t)-0.1\text{,} 
\ \ k\in \mathbb{N}_{0}= \{ 0,1,2,\dots \} \; .
\end{equation*}
By \eqref{2.3}, we see (Fig.~\ref{figure1}, b) that 
\begin{equation*}
g_{1}(t)=\sup_{s\leq t}\tau _{1}(s)=\left\{ 
\begin{array}{ll}
2k\text{,} & \text{if }t\in \left[ 2k+1,2k+7/3\right], \\ 
3t-4k-7\text{,} & \text{if }t\in \left[ 2k+7/3,2k+3\right],
\end{array}
\right. \text{ \ }k\in \mathbb{N}_{0}
\end{equation*}
and 
\begin{equation*}
g_{2}(t)=\sup_{s\leq t}\tau _{2}(s)=g_{1}(t)-0.1.
\end{equation*}

\begin{figure}[ptb]
\vspace{-20mm} \includegraphics[scale=0.75]{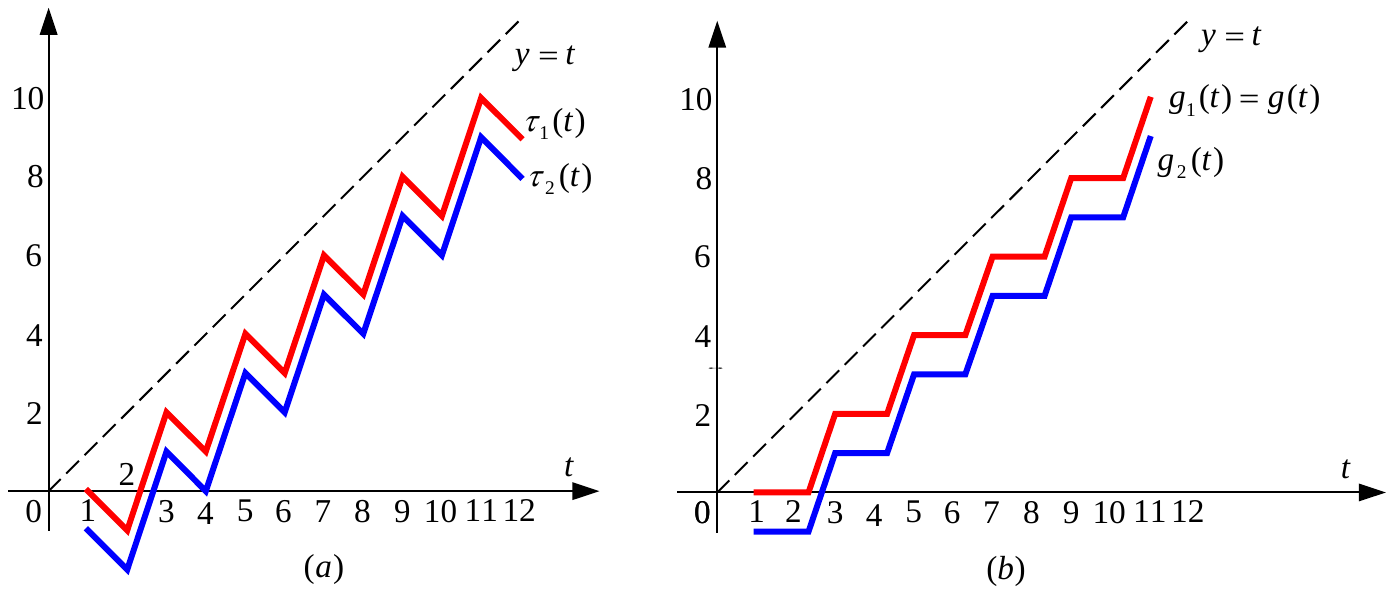} \vspace{-155mm%
}
\caption{The graphs of a) $\protect\tau_i(t)$ and b) $g_{i}(t)$}
\label{figure1}
\end{figure}

Therefore, in view of \eqref{2.4}, we have 
\begin{equation*}
g(t)=\max_{1\leq i\leq 2}\left\{ g_{i}(t)\right\} =g_{1}(t).
\end{equation*}
Define the function $f_{r}:[1,+\infty )\rightarrow (0,+\infty )$ as 
\begin{equation*}
f_{r}(t)=\int_{g(t)}^{t}\sum_{i=1}^{2}p_{i}(\zeta )a_{r}(g(t),\tau
_{i}(\zeta ))\,d\zeta.
\end{equation*}
Now, at $t=2k+3$, $k\in \mathbb{N}_{0}$, we have $g(t=2k+3)=2k+2$. Thus 
\begin{eqnarray*}
&&f_{1}(t=2k+3)=\int_{2k+2}^{2k+3}\sum_{i=1}^{2}p_{i}(\zeta )a_{1}(2k+2,\tau
_{i}(\zeta ))\,d\zeta \\
&=&\int_{2k+2}^{2k+3}\left[ p_{1}(\zeta )a_{1}(2k+2,\tau _{1}(\zeta
))+p_{2}(\zeta )a_{1}(2k+2,\tau _{2}(\zeta ))\right] \,d\zeta \\
&=&\frac{1}{2e}\int_{2k+2}^{2k+3}\exp \left\{ \int_{\tau _{1}(\zeta
)}^{2k+2}\left( p_{1}(\xi )+p_{2}(\xi )\right) d\xi \right\} d\zeta \\
&&+\frac{1}{2.2e}\int_{2k+2}^{2k+3}\exp \left\{ \int_{\tau _{2}(\zeta
)}^{2k+2}\left( p_{1}(\xi )+p_{2}(\xi )\right) d\xi \right\} d\zeta \\
&=&\frac{1}{2e}\int_{2k+2}^{2k+3}\exp \left\{ \frac{2.1}{2.2e}\int_{-\zeta
+4k+1}^{2k+2}d\xi \right\} d\zeta +\frac{1}{2.2e}\int_{2k+2}^{2k+3}\exp
\left\{ \frac{2.1}{2.2e}\int_{-\zeta +4k+0.9}^{2k+2}d\xi \right\} d\zeta \\
&=&\frac{1}{2e}\int_{2k+2}^{2k+3}\exp \left\{ \frac{2.1}{2.2e}\left( \zeta
-2k+1\right) \right\} d\zeta +\frac{1}{2.2e}\int_{2k+2}^{2k+3}\exp \left\{ 
\frac{2.1}{2.2e}\left( \zeta -2k+1.1\right) \right\} d\zeta \\
&=&\frac{11}{21}\left[ \exp \left\{ \frac{2.1}{2.2e}\cdot 4\right\} -\exp
\left\{ \frac{2.1}{2.2e}\cdot 3\right\} \right] +\frac{10}{21}\left[ \exp
\left\{ \frac{2.1}{2.2e}\cdot 4.1\right\} -\exp \left\{ \frac{2.1}{2.2e}%
\cdot 3.1\right\} \right] \\
&\simeq & 1.22696
\end{eqnarray*}
and therefore 
\begin{equation*}
\limsup_{t\rightarrow \infty }f_{1}(t)\gtrapprox 1.22696 >1.
\end{equation*}
That is, condition \eqref{2.10} of Theorem~\ref{theorem2.4} is satisfied for 
$r=1$, and therefore all solutions of \eqref{ex_3_eq3.1} oscillate.

Observe, however, that 
\begin{eqnarray*}
\liminf_{t\rightarrow \infty }\int\limits_{\tau _{\max
}(t)}^{t}\sum\limits_{i=1}^{m}p_{i}(s)ds &=&\liminf_{t\rightarrow \infty
}\int\limits_{\tau _{1}(t)}^{t}\sum\limits_{i=1}^{2}p_{i}(s)ds \\
&=&\left( \frac{1}{2e}+\frac{1}{2.2e}\right) \liminf_{t\rightarrow \infty
}\left( t-\tau _{1}(t)\right) =\frac{2.1}{2.2e}\cdot 1<\frac{1}{e}\text{,}
\end{eqnarray*}%
\begin{eqnarray*}
\liminf_{t\rightarrow \infty }\sum\limits_{i=1}^{m}p_{i}(t)\left( t-\tau
_{i}(t)\right) &=&\liminf_{t\rightarrow \infty }\left[ \frac{1}{2e}\left(
t-\tau _{1}(t)\right) +\frac{1}{2.2e}\left( t-\tau _{2}(t)\right) \right] \\
&=&\frac{1}{2e}\cdot 1+\frac{1}{2.2e}\cdot 1.1=\frac{1}{e},
\end{eqnarray*}%
and therefore none of conditions \eqref{1.8} and \eqref{1.10} is
satisfied.
\end{example}

\begin{example}
Consider the advanced differential equation 
\begin{equation}
x^{\prime }(t)-\frac{7}{40}x(\sigma _{1}(t))-\frac{7}{40}x(\sigma
_{2}(t))=0, \text{ \ \ \ }t\geq 1\text{,}  \label{3.2}
\end{equation}
where (see Fig.~\ref{figure2}, a) 
\begin{equation*}
\sigma _{1}(t)=\left\{ 
\begin{array}{ll}
4t-6k-2, & \text{if }t\in \left[ 2k+1,2k+2\right], \\ 
-2t+6k+10, & \text{if }t\in \left[ 2k+2,2k+3\right],
\end{array}%
\right. \text{ \ and \ }\sigma _{2}(t)=\sigma _{1}(t)+0.1\text{, \ \ \ }k\in 
\mathbb{N}_{0}\text{.}
\end{equation*}
By \eqref{2.15}, we see (Fig.~\ref{figure2}, b) that 
\begin{equation*}
\rho _{1}(t):=\inf_{t\leq s}\sigma _{1}(s)=\left\{ 
\begin{array}{ll}
4t-6k-2, & \text{if }t\in \left[ 2k+1,2k+1.5\right], \\ 
2k+4, & \text{if }t\in \left[ 2k+1.5,2k+3\right],
\end{array}%
\right. \text{ \ \ \ \ }k\in \mathbb{N}_{0}
\end{equation*}
and 
\begin{equation*}
\rho _{2}(t)=\inf_{t\leq s}\sigma _{2}(s)=\rho _{1}(t)+0.1\text{.}
\end{equation*}
Therefore, \eqref{2.16} gives 
\begin{equation*}
\rho (t)=\min_{1\leq i\leq 2}\left\{ \rho _{i}(t)\right\}
=\rho_{1}(t)=\left\{ 
\begin{array}{ll}
4t-6k-2, & \text{if }t\in \left[ 2k+1,2k+1.5\right], \\ 
2k+4, & \text{if }t\in \left[ 2k+1.5,2k+3\right],
\end{array}
\right. \text{ \ \ \ \ }k\in \mathbb{N}_{0}\text{.}
\end{equation*}

\begin{figure}[ptb]
\vspace{-20mm} \includegraphics[scale=0.75]{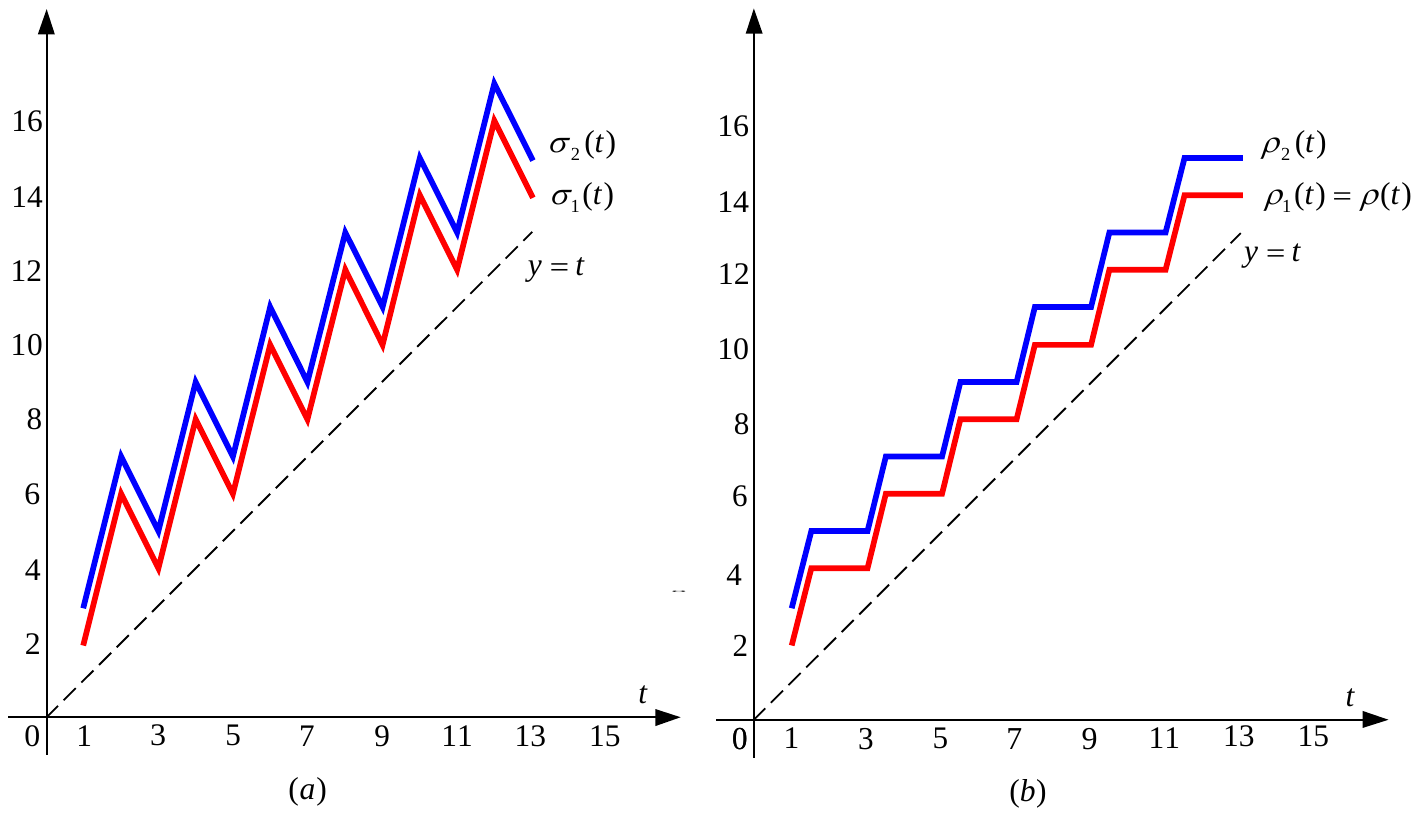} \vspace{-135mm}
\caption{The graphs of a) $\protect\sigma_i(t)$ and b) $\protect\rho_{i}(t)$}
\label{figure2}
\end{figure}


Define the function $f_{r}:[1,+\infty )\rightarrow (0,+\infty )$ as 
\begin{equation*}
f_{r}(t)=\int_{t}^{\rho (t)}\sum_{i=1}^{2}p_{i}(\zeta )b_{r}(\rho (t),\sigma
_{i}(\zeta ))\,d\zeta \text{.}
\end{equation*}
Now, at $t=2k+1$, $k\in \mathbb{N}_{0}$, we have $\rho (t=2k+1)=2k+2$. Thus 
\begin{eqnarray*}
&&f_{1}(t=2k+1)=\int_{2k+1}^{2k+2}\sum_{i=1}^{2}p_{i}(\zeta
)b_{1}(2k+2,\sigma _{i}(\zeta ))\,d\zeta \\
&=&\int_{2k+1}^{2k+2}\left[ p_{1}(\zeta )b_{1}(2k+2,\sigma _{1}(\zeta
))+p_{2}(\zeta )b_{1}(2k+2,\sigma _{2}(\zeta ))\right] d\zeta \\
&=&\frac{7}{40}\int_{2k+1}^{2k+2}\exp \left\{ \frac{7}{20}%
\int_{2k+2}^{4\zeta -6k-2}d\xi \right\} d\zeta +\frac{7}{40}%
\int_{2k+1}^{2k+2}\exp \left\{ \frac{7}{20}\int_{2k+2}^{4\zeta -6k-1.9}d\xi
\right\} d\zeta \\
&=&\frac{7}{40}\int_{2k+1}^{2k+2}\exp \left\{ \frac{7}{20}\left( 4\zeta
-8k-4\right) \right\} d\zeta +\frac{7}{40}\int_{2k+1}^{2k+2}\exp \left\{ 
\frac{7}{20}\left( 4\zeta -8k-3.9\right) \right\} d\zeta \\
&=&\frac{1}{8}\left[ \exp \left( 1.4\right) -1\right] +\frac{1}{8}\left[
\exp \left( 1.435\right) -\exp \left( 0.035\right) \right] \simeq 0.777403
\text{.}
\end{eqnarray*}
\begin{eqnarray*}
&&f_{2}(t=2k+1)=\int_{2k+1}^{2k+2}\sum_{i=1}^{2}p_{i}(\zeta
)b_{2}(2k+2,\sigma _{i}(\zeta ))\,d\zeta \\
&=&\int_{2k+1}^{2k+2}\left[ p_{1}(\zeta )b_{2}(2k+2,\sigma _{1}(\zeta
))+p_{2}(\zeta )b_{2}(2k+2,\sigma _{2}(\zeta ))\right] d\zeta \\
&=&\frac{7}{40}\int_{2k+1}^{2k+2}b_{2}(2k+2,\sigma _{1}(\zeta ))d\zeta +%
\frac{7}{40}\int_{2k+1}^{2k+2}b_{2}(2k+2,\sigma _{2}(\zeta ))d\zeta \\
&=&\frac{7}{40}\int_{2k+1}^{2k+2}\exp \left\{ \int_{2k+2}^{\sigma _{1}(\zeta
)}\left[ \frac{7}{40}b_{1}(2k+2,\sigma _{1}(\xi ))+\frac{7}{40}%
b_{1}(2k+2,\sigma _{2}(\xi ))\right] d\xi \right\} d\zeta \\
&&+\frac{7}{40}\int_{2k+1}^{2k+2}\exp \left\{ \int_{2k+2}^{\sigma _{2}(\zeta
)}\left[ \frac{7}{40}b_{1}(2k+2,\sigma _{1}(\xi ))+\frac{7}{40}%
b_{1}(2k+2,\sigma _{2}(\xi ))\right] d\xi \right\} d\zeta \\
&=&\frac{7}{40}\int_{2k+1}^{2k+2}\exp \left\{ \int_{2k+2}^{\sigma _{1}(\zeta
)}\left[ \frac{7}{40}\exp \left( \frac{7}{20}\int_{2k+2}^{\sigma _{1}(\xi
)}du\right) +\frac{7}{40}\exp \left( \frac{7}{20}\int_{2k+2}^{\sigma
_{2}(\xi )}du\right) \right] d\xi \right\} d\zeta \\
&&+\frac{7}{40}\int_{2k+1}^{2k+2}\exp \left\{ \int_{2k+2}^{\sigma _{2}(\zeta
)}\left[ \frac{7}{40}\exp \left( \frac{7}{20}\int_{2k+2}^{\sigma _{1}(\xi
)}du\right) +\frac{7}{40}\exp \left( \frac{7}{20}\int_{2k+2}^{\sigma
_{2}(\xi )}du\right) \right] d\xi \right\} d\zeta \\
&=&\frac{7}{40}\int_{2k+1}^{2k+2}\exp \left\{ \int_{2k+2}^{\sigma _{1}(\zeta
)}\left[ \frac{7}{40}\exp \left( \frac{7}{20}\left( \sigma _{1}(\xi
)-2k-2\right) \right) +\frac{7}{40}\exp \left( \frac{7}{20}\left( \sigma
_{2}(\xi )-2k-2\right) \right) \right] d\xi \right\} d\zeta \\
&&+\frac{7}{40}\int_{2k+1}^{2k+2}\exp \left\{ \int_{2k+2}^{\sigma _{2}(\zeta
)}\left[ \frac{7}{40}\exp \left( \frac{7}{20}\left( \sigma _{1}(\xi
)-2k-2\right) \right) +\frac{7}{40}\exp \left( \frac{7}{20}\left( \sigma
_{2}(\xi )-2k-2\right) \right) \right] d\xi \right\} d\zeta
\\
&=&\frac{7}{40}\int_{2k+1}^{2k+2}\exp \left\{ \int_{2k+2}^{\sigma _{1}(\zeta
)}\left[ \frac{7}{40}\exp \left( \frac{7}{20}\left( -2\xi +4k+8\right)
\right) +\frac{7}{40}\exp \left( \frac{7}{20}\left( -2\xi +4k+8.1\right)
\right) \right] d\xi \right\} d\zeta \\
&&+\frac{7}{40}\int_{2k+1}^{2k+2}\exp \left\{ \int_{2k+2}^{\sigma _{2}(\zeta
)}\left[ \frac{7}{40}\exp \left( \frac{7}{20}\left( -2\xi +4k+8\right)
\right) +\frac{7}{40}\exp \left( \frac{7}{20}\left( -2\xi +4k+8.1\right)
\right) \right] d\xi \right\} d\zeta \\
&=&\frac{7}{40}\int_{2k+1}^{2k+2}\exp \left\{ 
\begin{array}{c}
-\frac{1}{4}\left[ \exp \left( \frac{7}{20}\left( -8\zeta +16k+12\right)
\right) -\exp \left( 1.4\right) \right] \\ 
-\frac{1}{4}\left[ \exp \left( \frac{7}{20}\left( -8\zeta +16k+12.1\right)
\right) -\exp \left( 1.435\right) \right]%
\end{array}%
\right\} d\zeta \\
&&+\frac{7}{40}\int_{2k+1}^{2k+2}\exp \left\{ 
\begin{array}{c}
-\frac{1}{4}\left[ \exp \left( \frac{7}{20}\left( -8\zeta +16k+11.8\right)
\right) -\exp \left( 1.4\right) \right] \\ 
-\frac{1}{4}\left[ \exp \left( \frac{7}{20}\left( -8\zeta +16k+11.9\right)
\right) -\exp \left( 1.435\right) \right]%
\end{array}%
\right\} d\zeta \simeq 1.558893 > 1.
\end{eqnarray*}
Thus 
condition \eqref{2.19} of Theorem~\ref{theorem2.4a} is satisfied for $r=2$,
and therefore all solutions of \eqref{3.2} oscillate.

Observe, however, that 
\begin{eqnarray*}
\liminf_{t\rightarrow \infty }\sum\limits_{i=1}^{m}\int\limits_{t}^{\sigma
_{\min }(t)}p_{i}(s)ds &=&\liminf_{t\rightarrow \infty
}\sum\limits_{i=1}^{2}\int\limits_{t}^{\sigma _{1}(t)}p_{i}(s)ds \\
&=&\left( \frac{7}{40}+\frac{7}{40}\right) \liminf_{t\rightarrow \infty
}\left( \sigma _{1}(t)-t\right) =\frac{7}{20}<\frac{1}{e}\text{,}
\end{eqnarray*}
\begin{eqnarray*}
\liminf_{t\rightarrow \infty }\sum\limits_{i=1}^{m}p_{i}(t)\left( \sigma
_{i}(t)-t\right) &=&\liminf_{t\rightarrow \infty }\left[ \frac{7}{40}\left(
\sigma _{1}(t)-t\right) +\frac{7}{40}\left( \sigma _{2}(t)-t\right) \right]
\\
&=&\frac{7}{40}\cdot 1+\frac{7}{40}\cdot 1.1=0.3675<\frac{1}{e},
\end{eqnarray*}%
and therefore none of conditions \eqref{1.9} and \eqref{1.11} is
satisfied.
\end{example}


\section{Acknowledgments}

E. Braverman was partially supported by the NSERC research grant RGPIN-2015-05976.

\section{Competing interests}

The authors declare that they have no competing interests.

\section{Author's contributions}

The authors declare that they have made equal contributions to the paper.


\begin{thebibliography}{99}
\bibitem{ABBD} R.P. Agarwal, L. Berezansky, E. Braverman and A.
Domoshnitsky, Nonoscillation Theory of Functional Differential Equations
with Applications, Springer, New York, 2012.

\bibitem{1} F.\,A. Bartha, \'{A}. Garab and T. Krisztin, Local stability
implies global stability for the 2-dimensional Ricker map, \emph{\ J.
Difference Equ. Appl.} \textbf{19} (2013), 2043--2078.

\bibitem{2} L. Berezansky, E. Braverman and S. Pinelas, On nonoscillation of
mixed advanced-delay differential equations with positive and negative
coefficients, \textit{Comput. Math. Appl.} \textbf{58} (2009), 766--775.

\bibitem{3} E. Braverman and B. Karpuz, On oscillation of differential and
difference equations with non-monotone delays, \textit{Appl. Math. Comput.,} 
\textbf{218} (2011) 3880--3887.

\bibitem{4} G. E. Chatzarakis and \"{O}. \"{O}calan, Oscillations of
differential equations with several non-monotone advanced arguments, \textit{%
Dynamical Systems: An International Journal, }DOI:
10.1080/14689367.2015.1036007, (2015), 14 pages.


\bibitem{5} A. Elbert and I. P. Stavroulakis, Oscillations of first order
differential equations with deviating arguments, Univ of Ioannina T. R. No
172 (1990), Recent trends in differential equations, 163--178, World Sci.
Ser. Appl. Anal., 1, World Sci. Publishing Co. (1992).

\bibitem{6} L. E. Elsgolts, \textit{Introduction to the theory of
differential equations with deviating arguments, }Translated from the
Russian by R. J. McLaughlin,\textit{\ }Holden-Day, Inc., San Francisco,
Calif. - London - Amsterdam, 1966.

\bibitem{7} L.\,H. Erbe, Q.\,K. Kong and B.G. Zhang, Oscillation Theory for
Functional Differential Equations, Marcel Dekker, New York, 1995.

\bibitem{8} L. H. Erbe and B. G. Zhang, Oscillation of first order linear
differential equations with deviating arguments, Differential Integral
Equations, 1 (1988), 305-314.

\bibitem{9} N. Fukagai and T. Kusano, Oscillation theory of first order
functional-differential equations with deviating arguments, \textit{Ann.
Mat. Pura Appl. }\textbf{136}\textit{\ }(1984), 95--117.

\bibitem{10} B.\,R. Hunt and J.\,A. Yorke, When all solutions of $x^{\prime
}(t)=-\sum q_i(t) x(t-T_i(t))$ oscillate, \emph{J. Differential Equations} 
\textbf{53} (1984), 139--145.

\bibitem{11} R.\,G. Koplatadze and T.\,A. Chanturija, Oscillating and monotone
solutions of first-order differential equations with deviating argument,
(Russian), \textit{Differentsial'nye Uravneniya} \textbf{18}
(1982), 1463--1465, 1472.

\bibitem{12} R.\,G. Koplatadze and G. Kvinikadze, On the oscillation of
solutions of first order delay differential inequalities and equations,
Georgian Math. J., \textbf{3} (1994), 675--685.

\bibitem{13} M.\,R. Kulenovic and M.\,K. Grammatikopoulos, Some comparison
and oscillation results for first-order differential equations and
inequalities with a deviating argument, \textit{J. Math. Anal. Appl.} 
\textbf{131} (1988), 67--84.

\bibitem{14} T. Kusano, On even-order functional-differential equations with
advanced and retarded arguments, \emph{\textit{J. Differential Equations }}%
\textbf{45} (1982), 75--84.

\bibitem{15} G. Ladas and I.\,P. Stavroulakis, Oscillations caused by several
retarded and advanced arguments, \emph{\textit{J. Differential Equations }}%
\textbf{44} (1982), 134--152.

\bibitem{16} G.\,S. Ladde, Oscillations caused by retarded perturbations of
first order linear ordinary differential equations, \emph{\textit{Atti Acad.
Naz. Lincei Rendiconti }}\textbf{63} (1978), 351--359.

\bibitem{17a} G.\,S. Ladde, V. Lakshmikantham and B.\,G. Zhang, Oscillation Theory
of Differential Equations with Deviating Arguments, Monographs and Textbooks
in Pure and Applied Mathematics, vol. 110, Marcel Dekker, Inc., New York,
1987.

\bibitem{17} X. Li and D. Zhu, Oscillation and nonoscillation of advanced
differential equations with variable coefficients, \textit{J. Math. Anal.
Appl.} \textbf{269} (2002), 462--488.

\bibitem{18} H. Onose, Oscillatory properties of the first-order
differential inequalities with deviating argument, \textit{Funkcial. Ekvac.} 
\textbf{26} (1983), 189--195.

\bibitem{Stavro_AMC_2014} I.\,P. Stavroulakis, Oscillation criteria for
delay and difference equations with non-monotone arguments, \emph{Appl.
Math. Comput.} \textbf{226} (2014), 661--672.

\bibitem{19} X.H. Tang, Oscillation of first order delay differential
equations with distributed delay, J. Math. Anal. Appl. 289 (2004), 367-378.

\bibitem{20a} J. S. Yu, Z. C. Wang, B. G. Zhang and X. Z. Qian, Oscillations
of differential equations with deviating arguments, Panamer. Math. J. 
\textbf{2} (1992), 
59--78.

\bibitem{20} B. G. Zhang, Oscillation of solutions of the first-order
advanced type differential equations, \emph{\textit{Science exploration. }}%
\textbf{2} (1982), 79--82

\bibitem{21} D. Zhou, On some problems on oscillation of functional
differential equations of first order, \emph{J. Shandong University} \textbf{25} (1990), 434--442.
\end{thebibliography}
\end{document}